\documentclass[a4paper,12pt]{article}  
\usepackage[utf8]{inputenc}  
\usepackage[T2A]{fontenc}  
\usepackage[english]{babel}  
\usepackage[top=20mm, left=20mm, right=20mm, bottom=20mm]{geometry}  
\usepackage[pagebackref]{hyperref} 
\usepackage{sectsty}

\subsectionfont{\normalsize}

\usepackage{amsmath} 
\usepackage{amssymb} 
\usepackage{amsthm} 
\usepackage{amsfonts}
\usepackage{mathrsfs} 
\usepackage{euscript} 
\usepackage{bm, enumitem}

\usepackage{tikz}
    \usetikzlibrary{arrows.meta}
    \usetikzlibrary{calc}
    
    \tikzset{gdst/.style=
    {circle, draw=black!50, very thick, minimum height=1.2cm, inner sep=2pt, text centered, }, }

\usepackage{float, graphicx}
\graphicspath{{pictures/}}
\usepackage{color}

\DeclareMathSymbol{\sm}{\mathbin}{AMSa}{"39}

\newtheorem{thm}{\bf Theorem}

\newtheorem{lem}[thm]{\bf Lemma}

\newtheorem{dfn}{\bf Definition}
\theoremstyle{definition}

\newcommand{\beq}{\begin{equation}}
\newcommand{\eeq}{\end{equation}}

\title{On metric properties of self-affine polygonal dendrites\thanks{The study was carried out under the state contract of the Sobolev Institute of Mathematics (project FWNF-2026-0026).}}
\author{
    Andrei Tetenov$^\text{\dag}$\\ 
    \href{mailto:a.tetenov@gmail.com}    
         {\textsf{a.tetenov@gmail.com}}
    \and 
    Ivan Yudin$^\text{\dag}$\\ 
    \href{mailto:uivan566@gmail.com}
         {\textsf{uivan566@gmail.com}}
    \and 
    Kutlimuratov Dilmurat$^\text{\ddag}$\\ \href{mailto: dilmurat87@internet.ru}
         {\textsf{dilmurat87@internet.ru}}
}
\date{
$^\text{\dag}$Sobolev Institute of Mathematics, Novosibirsk, Russia\\
$^\text{\ddag}$Novosibirsk State University, Novosibirsk, Russia\\\;\\
\today}

\begin{document}

\maketitle
\noindent\rule{\textwidth}{1pt}
\begin{abstract}
 The main result of the paper is Theorem 3, which states that for any self-affine dendrite $K$, generated by a polygonal system, there are constants $C>0$ and $\lambda\in (0,1)$ such that for any $x,y\in K$, the Jordan arc $\gamma\subset K$ with endpoints $x,y$ satisfies inequality $\dfrac{{\rm diam}(\gamma)}{\|x-y\|^\lambda}\le C$.
\end{abstract}
\noindent\rule{\textwidth}{1pt}\\

 In 2017 \cite{STV} it was proved  that if a self-similar dendrite $K\subset {\mathbb R}^2$ is generated by a system of convex polygons, then $K$ is a continuum with $C$-bounded turning \cite{V}. That is, there is a constant $C>0$ such that for any $x,y\in K$, the Jordan arc $\gamma\subset K$ with endpoints $x,y$ satisfies the inequality $\dfrac{{\rm diam}(\gamma)}{\|x-y\|}\le C$.  This statement is not valid for self-affine dendrites in ${\mathbb R}^2$,  even if they are generated by a system of convex polygons.\\

\section{Preliminaries}
{\bf Self-similar sets.} Let $\mathcal{S}=\{S_1,\ldots, S_m\}$ be a system of contractive maps in $\mathbb{R}^n$.  
A  nonempty compact  $K$ that satisfies the equation $K=\bigcup\limits_{i=1}^m S_i(K)$ is called a self-similar set or the {\em attractor} of the system $\mathcal{S}$.
If the maps $S_i\in \mathcal{S}$ are contractive similarities (or contractive affine maps), then the attractor $K$ is a self-similar (resp. self-affine) set.\\

\noindent{\bf The system of copies of $K$.} The set $I=\{1,...,m\}$ is called the set of indices for the system $ \mathcal S$.
The maps $S_1,\ldots, S_m$ generate the semigroup $\mathcal G(\mathcal S)$ that consists of compositions $S_{i_1i_2...i_n}=S_{i_1}S_{i_2}...S_{i_n}$, labeled by {\em multiindices} ${\bf i}={i_1i_2...i_n}$.  The set of all multiindices of length $n$ is $I^n$ and $I^*=\bigcup\limits_{n=1}^\infty I^n$ is the set of all multiindices. We write ${\bf i}\sqsubset {\bf j}$, if ${\bf i}$  is an initial word in $ {\bf j}$. In that case $K_{\bf j}$ is a subcopy of $K_{\bf i}$.  If ${\bf i}\not\sqsubset {\bf j}$ and ${\bf j}\not\sqsubset {\bf i}$, then ${\bf i}$  and $ {\bf j}$ are called {\em incomparable}.\\
If ${\bf i}\in I^n$, the sets $K_{\bf i}=S_{\bf i}(K)$ are called {\em copies} of $K$ of order $n$.\\
Equation $T(B)=\bigcup\limits_{i=1}^m S_i(B)$ defines the {\em Hutchinson operator} $T$ of the system $\mathcal S$. By Hutchinson's theorem \cite{Hut81}, for any compact set $B$, the sequence $T^n(B)$ converges to the attractor $K$. In particular, if $B\subset T(B)$, then $B\subset K$.

\noindent{\bf The SIP and self-similar boundary.} The system $\mathcal S$  has the single-point  intersection property (SIP) if,  for any pair $K_i,K_j$ of the copies of the attractor $K$, the intersection $K_i\cap K_j$ is either empty or is a singleton $\{A_{ij}\}$.  The set of all intersection points $A_{ij}$ is denoted by $\mathcal C$.\\ 
The set of all  $x\in K$  such that for some ${\bf i}\in I^*$, $S_{\bf i}(x)\in \mathcal C$, is called the {\em self-similar boundary} $\partial K$ of the set $K$.\\
Note that if the system $\mathcal S$ has the SIP, then for any two incomparable
multiindices ${\bf i,j}\in I^* $
   the intersection $K_{\bf i}\cap K_{\bf j}$ is empty or is a singleton.

\noindent{\bf The intersection graph and dendrite criterion.} In the SIP holds, one can consider
{\em the bipartite intersection graph} $\widehat G=\widehat G(\mathcal{S})$ of the system $\mathcal{S}$. It is defined as a bipartite graph with parts $\{K_i:\; i\in I\}$ and $\mathcal C$, and a set of edges $\{(K_i,A):A\in \mathcal C\cap K_i\}$.\\

\noindent\includegraphics[width=\linewidth]{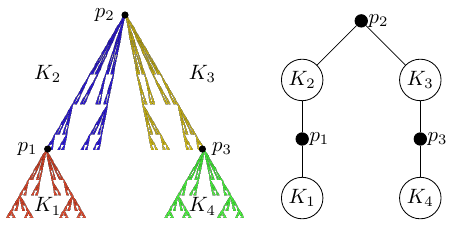}
{ \small {\bf Figure 1.} The attractor $K$ on the left has 4 copies. The bipartite intersection graph is shown on the right.}\\

\begin{dfn} 
A dendrite is a locally connected continuum that does not contain simple closed arcs.
\end{dfn}

The following criterion that was proved in  {\cite{FIP}} makes it possible to distinguish self-similar dendrites that have the single-point intersection property.

\begin{thm}\label{denfip}
Let $\mathcal{S}=\{S_1,\ldots, S_m\}$ be a system of injective contractions in $\mathbb{R}^n$.
If the attractor $K=K(\mathcal{S})$ has the single intersection property and 
 the bipartite intersection graph  $\widehat G(\mathcal{S})$ of the system $\mathcal{S}$ is a tree, then  $K$ is a dendrite.
\end{thm}

\section{Polygonal systems}

Extending the idea of the paper \cite{STV} to the self-affine case, we introduce the construction that produces a wide variety of self-affine  dendrites possessing the SIP; we prove the  remarkable metric properties of such dendrites.

\begin{dfn}
  Let $P$ be a convex polygon in $\mathbb R^2$ with the set of vertices $V=\{A_1,...,A_n\}$. Let $\mathcal S=\{S_1, S_2, \ldots, S_m\}$ be a system  of affine contractive maps  that  satisfy the following conditions: \\
  1) For any $i\in I$, $S_i(P)\subset P$; \\
  2) $\bigcup\limits_{i \in I } S_i(V) \supset V$; \\
  3) For any $i\neq j$, $S_i(P)\bigcap S_j(P)=S_i(V)\bigcap S_j(V)$ and $\#S_i(V)\bigcap S_j(V)\le1$;\\
  4)  The set $\bigcup\limits_{i\in I}  S_i(P) $ is simply connected.\\ 
  Then the system $\mathcal S$ is called a {\em simply connected self-affine $P$-polygonal system}.
\end{dfn}

Further in this article,  we will use the short term "polygonal system" keeping the longer name in mind.

{\centering\includegraphics[width=.6\linewidth]{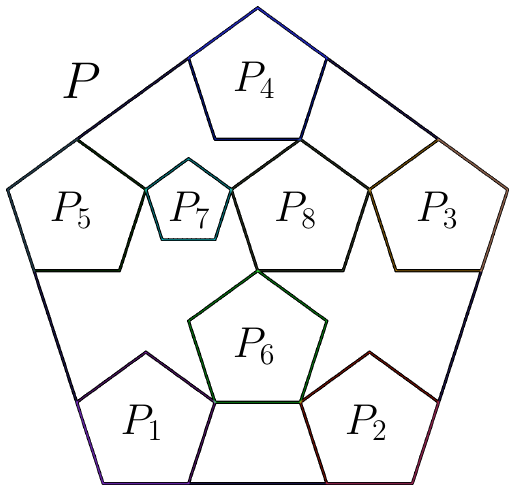}\\
{ \small {\bf Figure 2.} A polygonal system with 8 copies.The union of small polygons $P_1-P_8$ is simply connected.}}\\

For a polygonal system, we denote $P_{\bf i}=S_{\bf i}(P)$, $V_{\bf i}=S_{\bf i}(V)$.
By definition, $P\supset K$, therefore, $P_{\bf i}\subset K_{\bf i}$. Condition 2) means that $T(V)\supset V$, and consequently, $V\subset K$.
Then,  
$P_{\bf i}\subset P_{\bf j}$ iff $K_{\bf i}\subset K_{\bf j}$, and if ${\bf i}$ and ${\bf j}$ are incomparable, $K_{\bf i}\cap K_{\bf j}=P_{\bf i}\cap P_{\bf j}=V_{\bf i}\cap V_{\bf j}$.

For simply connected polygonal systems, the following Theorem is true.

\begin{thm}\label{main1}
Let $\mathcal S = \{S_1,  S_2,  \ldots S_m\}$ be a simply connected self-affine  $P$-polygonal system in $\mathbb R^2$. Then the attractor $K$ of the system $\mathcal S$ is a self-affine dendrite and its  boundary   $\partial K$ is contained in the set $V$.
\end{thm}

{\bf Remark 1.} The idea of simply connected self-similar polygonal systems was introduced by the first author in \cite{STV}.  Theorem \ref{main1} was proved in \cite{STV} in the self-similar case. Here we give a more simple proof that applies to self-affine sets and is based on Theorem 1. 

 \begin{proof}
 
 1)  For any non-equal $i,j\in I$, $K_i\cap K_j=P_i\cap P_j=V_i\bigcap V_j$. If this intersection  is non-empty, it is a unique point which we denote by $A_{ij}$.
 For each convex polygon $P_i$ we take a point $O_i\in\dot P_i$ and connect  it with each of the points $A_{ij}\in P_i$ by a straight line segment $E_{ij}$. Thus, we obtain a bipartite graph $\Gamma$ with parts
 $\{O_i,i\in I\}$ and $\{A_{ij},i,j\in I\}$ and edges $E_{ij}$. This graph $\Gamma$ is isomorphic to the graph $\widehat G(\mathcal{S})$.

 For any polygon $P_i$, the  union of all segments $E_{ij}\subset P_i$ is a deformation retract of $P_i$.
 Therefore, the graph $\Gamma$ is a deformation retract of the set $\bigcup\limits_{i\in I}  S_i(P) $. 
 This implies that  the graph $\Gamma$ is a tree, and consequently, the graph  $\widehat G(\mathcal{S})$ is a tree.
 From  Theorem 1  it follows that $K$ is a dendrite.\\

 From Condition 2) it follows that
 if for some $x\in K$ and $i\in I$, $S_i(x)\in V$, then $x\in V$. Therefore, if for some $x\in K$ and ${\bf i}\in I$, $S_{\bf i}(x)\in V$, then $x\in V$.
 
 For any $A \in\mathcal{C}$, the point
 $A$ is the intersection of a finite number of polygons $P_{i_1},..,P_{i_k}$ and for any $j\notin\{{i_1},..,{i_k}\}$, $A\notin S_j(P)$. Consequently, if for some $x\in K$ and $i\in I$, $S_i(x)\in \mathcal{C}$, then $x\in V$. This shows that $\partial K\subset V$.
 \end{proof}

Since the attractor $K$ of the system $\mathcal S$ is a dendrite, for any points $x,y\in K$ there is a unique subarc $\gamma(x,y)\subset K$ with endpoints $x,y$.

In the self-similar case it was proved in \cite{STV} that the attractor $K$ of a simply-connected polygonal system $\mathcal{S}$ is a continuum with a $C$-bounded turning \cite{V}. It means that there is a constant $C$ such that for any two points $x,y\in K$, the arc $\gamma(x,y)\subset K$ satisfies the inequality ${\rm diam}\gamma(x,y)\le C\|x-y\|$. In particular, each arc $\gamma(x,y)\subset K$ is a quasi-arc. 

This is impossible for the attractors of self-affine polygonal systems.

For example, consider the images $S_1^n(\gamma(p_1,p_3))$ of the arc $\gamma(p_1,p_3)$ of a self-affine polygonal dendrite $K$ in Figure 1. 

The diameter of the nth image is
$D_n={\rm diam}(S_1^n(\gamma(p_1,p_3)))=\dfrac{2^{n+1}}{3^{n+1}}$, while the distance between its endpoints is $\Delta_n=\|S_1^n(p_1)-S_1^n(p_3)\|=\dfrac{2}{3^{n+1}}$. Thus, the ratio
$\dfrac{D_n}{\Delta_n}=2^n$ tends to infinity as n grows.

\section{The main Theorem}

An arc $\gamma$ is called a $t$-quasi-arc, $t > 1$, if there is a constant
$K > 0$ such that for every $x,y\in \gamma$ the subarc $\gamma(x,y)$ satisfies the inequality $ \left( {\rm diam}(\gamma(x,y))\right)^t\le K\|x-y\|$.

The main result of the paper is  that in the case of affine polygonal system $\mathcal{S}$   there is $\lambda\in (0,1)$ such that for any two points $x,y$ in its attractor $K$, they can be connected by a $1/\lambda$-quasiarc  \cite{N},\cite{Wen2003}.

It is essential that the polygon $P$ in the definition of the system is required to be convex,
because even in the self-similar case there is the example of a self-similar arc $K$ which is the attractor of a polygonal system with a non-convex polygon $P$ and which is not a $t$-quasiarc for any $t\ge 1$.\cite{Wen2003,ATK}.

\begin{thm}\label{main2}
Let $\mathcal S = \{S_1,  S_2,  \ldots S_m\}$ be a simply connected self-affine  $P$-polygonal system in $\mathbb R^2$ with the attractor $K$.

There are constants $C>0$ and $\lambda\in (0,1)$ such that for any $x,y\in K$ the subarc $\gamma(x,y)\subset K$ with endpoints $x,y$ satisfies the inequality
  $${\rm diam}(\gamma(x,y)) \le C|x-y|^{\lambda}.$$

\end{thm}

\begin{proof}
    
 For any ${\bf i}\in I^*$ define
 $$Q_{\bf i}=\max\left\{\dfrac{\|S_{\bf i}(x)-S_{\bf i}(y)\|}{x-y}, x\neq y, x,y\in P\right\}$$ and
 $$q_{\bf i}=\min\left\{\dfrac{\|S_{\bf i}(x)-S_{\bf i}(y)\|}{x-y}, x\neq y, x,y\in P\right\}$$
 $$\lambda=\min\limits_{i\in I}\dfrac{\log Q_i}{\log q_i}$$

 \begin{lem}
  For any multiindex  ${\bf i}=i_1...i_n\in I^\infty$,   
  $$Q_{\bf i}\le q_{\bf i}^\lambda$$    
 \end{lem}
\begin{proof}
 Note that all the mappings $S_i\in \mathcal{S}$ are contractions, thereby, all $Q_i$ and $q_i$ lie in $(0,1)$.
 
  Consequently, for any $i\in I$, $Q_i\le q_i ^\lambda$.

  For any two maps $S_i,S_j$ with respective coefficients $Q_i,Q_j,q_i,q_j$ it can be easily derived for their composition $S_{ij}=S_iS_j$ that
$Q_{ij}\le Q_i Q_j$ and
$q_{ij}\ge q_i q_j$.\\
 As a result, for any multiindex  ${\bf i}=i_1...i_n\in I^\infty$ we obtain that $$Q_{\bf i}\le Q_{i_1}Q_{i_2}...Q_{i_n}\le(q_{i_1}q_{i_2}...q_{i_n})^\lambda\le q_{\bf i}^\lambda$$ 
\end{proof}
 
 For convenience, we assume that $|P|=1$. In this case, $|P_i|\le Q_i\le q_i^\lambda$ for any $i\in I$ and $|P_{\bf i}|\le Q_{\bf i}\le q_{\bf i}^\lambda$ for any ${\bf i}\in I^*$.\\

\begin{lem} The  inequality
\begin{equation}\label{leC}
  {\rm diam}(\gamma(x,y))\le C\|x-y\|^\lambda   
 \end{equation}
  is preserved by all the elements 
  $S_{\bf i}$ of  the semigroup $\mathcal G(\mathcal S)$.  
\end{lem}
 \begin{proof}
  Take $x,y\in K$ and  let $\gamma(x,y)\subset K$  be the subarc with endpoints $x$ and $y$. Suppose that for some $C>0$,
  ${\rm diam}(\gamma(x,y))\le C\|x-y\|^\lambda$

 For any ${\bf i}\in I^*$, we have the inequalities
 $${\rm diam}(S_{\bf i}(\gamma(x,y)))\le Q_{\bf i}\  {\rm diam}(\gamma(x,y)) \mbox{,\ \  and}$$  
 $$\|S_{\bf i}(x)-S_{\bf i}(y)\|\ge q _{\bf i}\|x-y\|\mbox{, which imply }$$
  $${\rm diam}(S_{\bf i}(\gamma(x,y)))\le C\|S_{\bf i}(x)-S_{\bf i}(y)\|^\lambda.$$
     
 \end{proof}
 
To find the value of $C$,  the following values are required:\\

 1) $\rho=\min\{d_1,d_2\}$ where $d_1$ is a minimal distance between the points $x\in P_i$, $y\in P_j$, where $P_i\cap P_j=\varnothing$ and 
 $d_2$ is a minimal distance between the points $A\in V$ and $x\in P_i$, where $A\notin P_i$;\\
 
  2) $\beta$ - the minimum angle between the sides of  polygons  $P_i$ and $P_j$ that have a common endpoint at the vertex $A_{ij}$.\\

  Take some points $x',y'\in K$. Let $K_{\bf k}$ be the smallest copy of $K$ that contains $x'$ and $y'$.
  Then the points $x=S_{\bf k}^{-1}(x)$ and
  $y=S_{\bf k}^{-1}(y)$ are  contained in different copies, say $K_i$ and $K_j$, of the attractor $K$.
  
  There are 3 possibilities of relative positioning of $x\in K_i\subset P_i$ and $y\in K_j\subset P_j$: 

  1. $x$ and $y$ belong to non-intersecting polygons $P_i$ and $P_j$;

  2. $x=A\in V$, $A\notin P_j$, $y\in P_j$;

  3. $K_i\cap K_j=P_i\cap P_j=\{A_{ij}\}$, $x\in P_{i}$, $y\in P_{j}$.\\
  
 In the first case, 
 if for some $i,j\in\{1,2,...,m\}$, $ P_i\bigcap P_j=\varnothing$, then for any
  $x\in K_i, y\in K_j$, $\rho \le d(x,y)$. At the same time, ${\rm diam}(\gamma (x,y))\le |P|=1$, which implies the following inequality.
  $$\dfrac{{\rm diam}(\gamma(x,y))}{\|x-y\|^\lambda}\le\dfrac{1}{\rho ^\lambda}.$$
  
 In the second case, if $x\in K_i$, $A\in V $ and  $ A\not\in K_i$,
 then 
 ${\rm diam}(\gamma (x,A))\le |P|=1$ and $\|x-A\|\ge\rho$.  This implies the inequality 
 $$\dfrac{{\rm diam}(\gamma(x,A))}{\|x-A\|^\lambda}\le\dfrac{1}{\rho ^\lambda}.$$

 In the third case we have  the intersection $P_i\bigcap P_j=\{A_{ij}\}$. Take the points $x\in P_i$ and $ y\in P_j$. Let $P_{\bf i}$ 
 be the smallest polygon  that contains 
 $x$ and $A_{ij}$. There is at least one polygon $P_{{\bf i}k}$ that contains $x$. Similarly, let $P_{\bf j}$ 
 be the smallest polygon  that contains 
 $y$ and $A_{ij}$. There is at least one polygon $P_{{\bf j}l}$ that contains $y$. 
 Suppose for convenience that $q_{\bf i}\ge q_{\bf j}$
 
 Then $ d(x,A_{ij})\ge q_{\bf i}\rho $ and $d(y,A_{ij})\ge q_{\bf j}\rho $. Note that  the angle between the sides of the polygons  $P_{\bf i}$ and $P_{\bf j}$ having a common vertex $A_{ij}$ is at least $\beta$. Therefore, we get $ d(x,y)\ge q_{\bf i} \rho \sin\beta  $  and ${\rm diam}(\gamma(x,y))\le |P_{\bf i}|+|P_{\bf i}|\le 2 q_{\bf i}^\lambda$. Consequently,
  $$\dfrac{{\rm diam}(\gamma(x,y))}{\|x-y\|^\lambda}\le\dfrac{2q_{\bf i}^\lambda}{\rho^\lambda q_{\bf i}^\lambda(\sin\beta)^\lambda}=\dfrac{2}{\rho^\lambda(\sin\beta)^\lambda}.$$

This upper bound is valid for each of  cases 1,2,3.

If we apply a mapping $S_{\bf k}$ to 
$x$,$y$ and $\gamma(x,y)$, we obtain the relation
$$\dfrac{{\rm diam}(S_{\bf k}(\gamma(x,y))}{\|S_{\bf k}(x)-S_{\bf k}(y)\|^\lambda}\le \dfrac{Q_{\bf k}^n2q_{\bf k}^\lambda}{\rho^\lambda q_k^{(n+1)\lambda}(\sin(\beta))^\lambda}$$

Since $Q_{\bf k}^n\le q_k^{n\lambda}$, we finally obtain

$$\dfrac{{\rm diam}(\gamma(x',y'))}{\|x'-y'\|^\lambda}\le\dfrac{2}{\rho^\lambda (\sin(\beta))^\lambda}$$

Therefore, $C=\dfrac{2}{\rho^\lambda (\sin(\beta))^\lambda}$ is the desired constant.

\end{proof}

\vspace{0.15cm}

\noindent {\bf Acknowledgments} 
\vspace{0.15cm}
\noindent 
The work is supported by the Mathematical Center in Akademgorodok under the agreement No. 075-15-2025-348 with the Ministry of Science and Higher Education of the Russian Federation.


\begin{thebibliography}{99} 
\small

 \bibitem{AST}
Allabergenova,K., Samuel,M., Tetenov,A.(2024).
Intersections of the pieces of self-similar dendrites in the plane,
\textit{Chaos, Solitons and Fractals}182:114805
\url{https://doi.org/10.1016/j.chaos.2024.114805}


 \bibitem{ATK} 
     Aseev, V.V., Tetenov, A.V.,  Kravchenko, A.S. (2003). On Self-similar Jordan Curves on the Plane. \textit{Siberian Mathematical Journal} 44: 379--386 . \url{https://doi.org/10.1023/A:1023848327898}

\bibitem{Hata85} 
Hata, M. (1985).
On the structure of self-similar sets. 
\textit{Japan Journal of Applied Mathematics} 2(2): 381--414.
\url{https://doi.org/10.1007/bf03167083}

\bibitem{Hut81} 
Hutchinson, J. (1981).
Fractals and Self-Similarity. 
\textit{Indiana University Mathematics Journal} 30(5): 713--747.
\url{https://doi.org/10.1512/iumj.1981.30.30055}

\bibitem{N} Norton, A. (1989). Functions not constant on fractal quasi-arcs of critical points. \textit{Proceedings of the American Mathematical Society} 106(2): 397--405.
\url{https://doi.org/10.1007/BF02807200}
    
\bibitem{STV} {Samuel M., Tetenov A., Vaulin D.} (2017). Self-Similar Dendrites Generated by Polygonal Systems in the Plane.\textit{ Siberian Electronic Mathematical Reports}14:~737--751. \url{https://doi.org/10.17377/semi.2017.14.063}
   
\bibitem{FIP} 
Tetenov, A., Yudin, I., Kadirova, M. (2025).
Finiteness properties for self-similar continua.
\textit{Discrete and Continuous Dynamical Systems -- Series S.} 2025057: 1--13.
\url{https://doi.org/10.3934/dcdss.2025057}

\bibitem{V}
Väisälä, J.(1991) Bounded turning and quasiconformal maps. \textit{Monatshefte für Mathematik} 111: 233--244 \url{https://doi.org/10.1007/BF01294269}


\bibitem{Wen2003}
Wen, Zhi-Ying, and Li-Feng Xi. (2003). Relations among Whitney sets, self-similar arcs and quasi-arcs. \textit{Israel Journal of Mathematics} 136.1: 251--267.
\url{https://doi.org/10.1007/BF02807200}    



\end{thebibliography}
\end{document}